\documentclass{llncs}
\usepackage[utf8]{inputenc}
\usepackage{amsmath}
\usepackage{amsfonts}
\usepackage{amssymb}
\usepackage{bussproofs}
\usepackage{footmisc}
\usepackage{stmaryrd}
\usepackage{color}
\usepackage{mathtools}
\usepackage{prftree}
\usepackage{proof}
\usepackage{mathabx}
\usepackage{bussproofs}
\usepackage{relsize}
\usepackage{scalerel}
 \usepackage{microtype}
\usepackage[bbgreekl]{mathbbol}
\usepackage{lineno}
\usepackage{etex}
\newcommand{\LK}{\textbf{LK}}
\newcommand{\LKE}{\textbf{LKE}}
\newcommand{\LKS}{\textbf{LKS}}
\newcommand*{\shifttext}[2]{%
  \settowidth{\@tempdima}{#2}%
  \makebox[\@tempdima]{\hspace*{#1}#2}%
}
\newcommand\fat[1]{\ThisStyle{\ooalign{%
  \kern.46pt$\SavedStyle#1$\cr\kern.33pt$\SavedStyle#1$\cr%
  \kern.2pt$\SavedStyle#1$\cr$\SavedStyle#1$}}}

\makeatother

\author{David M. Cerna\inst{1} and Michael Lettmann\inst{2}}
\authorrunning{D.M. Cerna and M. Lettmann}

\institute{
  Reseach Institute for Symbolic Computation\\
Johannes Kepler University\\
Linz, Austria\\
\email{david.cerna@risc.jku.at}
\and
   Institute of Information Systems\\
   Technische Universit\"at Wien\\
    Vienna, Austria\\
    \email{lettmann@logic.at}
 }
 
\title{Integrating a Global Induction Mechanism into a Sequent Calculus}

\newcommand{\SchCom}[2]{\ensuremath{\left(\ #1 \ \mbox{{{\Large :}}} \ #2 \ \right)}}
\newcommand{\ComSep}[2]{\ \ensuremath{\displaystyle \prescript{#1}{#2}{\Big| }}\ }

\begin{document}
\maketitle
\begin{abstract}
Most interesting proofs in mathematics contain an inductive argument 
which requires an extension of the LK-calculus to formalize. The most
commonly used calculi for induction contain a separate rule or axiom which 
reduces the valid proof theoretic properties of the calculus. To the best of our
knowledge, there are no such calculi which allow cut-elimination to  
a normal form with the {\em subformula property}, i.e.\ every formula 
occurring in the proof is a subformula of the end sequent. Proof schemata are 
a variant of LK-proofs able to simulate induction by linking proofs together.
There exists a schematic normal form which has comparable proof theoretic behaviour 
to  normal forms with the subformula property. However, a calculus for the 
construction of proof schemata does not exist. In this paper, we introduce 
a calculus for proof schemata and prove soundness and completeness with respect 
to a fragment of the inductive arguments formalizable in {\em Peano arithmetic}. 
\end{abstract}

\section{Introduction}
The schematic construction of objects that form the basis of proof schemata as described in this paper was introduced by V. Aravantinos et al.~\cite{VAravantinos2009,VAravantinos2010,VAravantinos2011,VAravantinos2011a,VAravantinos2013,VAravantinos2011b}. Initially, they considered formulas of an indexed propositional logic with a single {\em free numeric parameter} and with two new logical connectors, i.e. $\vee$-iteration and $\wedge$-iteration. They developed a tableau based decision procedure for the satisfiability of a monadic\footnote{In this fragment the use of schematic constructors is restricted to one free parameter per formula.} fragment of this logic. An extension to a special case of multiple parameters was also investigated by D. Cerna~\cite{DCerna2014}. In a more recent work, V. Aravantinos et al.~\cite{VAravantinos2013} introduced a superposition resolution calculus for a clausal representation of indexed propositional logic. The calculus provided decidability results for an even larger fragment of the monadic fragment. The clausal form allows an easy extension to indexed predicate logic, though all decidability results are lost. In either case, the refutations producible by the calculus for unsatisfiable clause sets is quite restricted. 

Nonetheless, these results inspired investigations into the use of schemata as an alternative formalization of induction to be used for proof analysis and transformation. This is not the first attempt at an alternative formalization of induction, that is with respect to Peano arithmetic~\cite{GTakeuti1975}. However, all existing examples~\cite{JBrotherston2005,JBrotherston2010,RMcdowell1997}, to the best of our knowledge, lack proof normal forms with the {\em subformula property}\footnote{A proof fulfilling the subformula property can be referred to as {\em analytic}.}, i.e. every formula occurring in the proof is a subformula of a formula occurring in the end sequent. What we mean by this is that performing cut-elimination in the presence of induction, for example, to a proof results in a non-analytic proof: some part of the argument is not directly connected to the theorem being proven. Moreover, two important constructions extractable from proofs with the subformula property, {\em Herbrand sequents}~\cite{SHetzl2008,GTakeuti1975} and {\em expansion trees}~\cite{DMiller1987}, are not extractable from proofs within these calculi. While Herbrand sequents allow the representation of the propositional content of first-order proofs, expansion trees generalize Herbrand's theorem and, therefore, Herbrand sequents.

However, schematic formalisms seem to get around this problem. The first proof analysis carried out using a rudimentary schematic formalism was Baaz et al.~\cite{MBaaz2008a}, where proof analysis of F\"{u}rstenberg's proof of the infinitude of primes was successfully performed using CERES~\cite{MBaaz2000}(Cut-Elimination by RESolution). The formalism discussed in this paper is an extension of CERES introduced by C. Dunchev et al.~\cite{CDunchev2014}. It allows the extraction of a Herbrand sequent from the resulting normal form produced by cut-elimination in the presence of induction. Problematically, the method of cut-elimination introduced in~\cite{CDunchev2014} is not known to be complete, in terms of cut-elimination, and is very difficult to use. For an example of the difficulties see Cerna and Leitsch's work~\cite{DCerna2016}. A much improved version of this cut-elimination method has been introduced in~\cite{ALeitsch2017}. using the superposition resolution calculus of~\cite{VAravantinos2013}. The method is complete and always produces a Herbrand sequent, but
it is much weaker than the method of~\cite{CDunchev2014}. Its exact expressive power is not known, but as mentioned,  the superposition resolution calculus' expressive power is quite restricted. The method of~\cite{CDunchev2014}  can formalize proof normal forms with a non-elementary length with respect to the size of the end sequent\footnote{See Orevkov's proof~\cite{VOrevkov1991} or Boolos' proof~\cite{GBoolos1984}. }. 

Currently, proof schemata are defined as an ordered sequence of $\LKS$-proofs \cite{CDunchev2014}. A problem with this construction is that the set of all valid proof schemata is not well defined, nor is the set of all well-formed proof schemata. The $\LKS$-calculus introduces concepts such as {\em links} but does not place restrictions on what a sound application of the rule is, rather an additional construction, {\em proof schemata}, is needed to define sound application of the rule. This leads to a foundational theory which is confusing and hard to work with.  

In this work we present a novel calculus for proof schemata. This calculus implicitly enforces the sound  application of inferences. Moreover, we show completeness with respect to the  {\em $k$-induction} fragment of Peano arithmetic~\cite{ALeitsch2017}. We also introduce a normal form for proofs in our calculus called {\em pre-proof schema normal form} which allows an easy translation to proof schemata. As a conclusion, we discuss the usage of proof schemata as a method of proof storage as well as a few open problems such as multi-parameter extensions and the relationship between the calculus and the cut-elimination method of~\cite{CDunchev2014} in terms of normal form construction. 

The rest of this paper is as follows: In Section~\ref{sec:ProofSchema}, we discuss the necessary background knowledge needed for the results. In Section~\ref{sec:evalInter} we discuss the evaluation and interpretation of proof schemata. In Section~\ref{sec:calculus}, we introduce the concept of the calculus. In Section~\ref{sec:SoundComplete}, we show soundness and completeness of the calculus.   In Section~\ref{sec:Conclusion}, we conclude the paper and mention possible applications, future work, and open problems.

\section{Preliminaries}
\label{sec:ProofSchema}

In this section, we provide a formal construction of proof schemata.

\subsection{Schematic Language}

We work in a two-sorted version of classical first-order logic. The first sort we consider is $\omega$, in which every term normalizes to a {\em numeral}, that is a term inductively constructable by $N \Rightarrow s(N)\ |\ 0$, such that $s(N)\neq 0$ and $s(N)=s(N')\to N=N'$. We will denote numerals by lowercase greek letters, i.e. $\alpha$, $\beta$, $\gamma$, etc. Furthermore, the omega sort includes a countable set of parameter symbols $\mathcal{N}$. For this work, we will only need a single parameter symbol which in most cases we denote by $n$. We use  $k,k'$ to represent numeric expressions containing the parameter. The parameter symbol $n$ will be referred to as the {\em free parameter}.

The second sort $\iota$ (individuals) is a standard first-order term language extended by {\em defined function symbols} and {\em schematic variable symbols}. To distinguish defined and uninterpreted function symbols we partition the functions of $\iota$ into two categories, {\em uninterpreted function symbols} $\mathbf{F_{u}}$ and {\em defined function symbols} $\mathbf{F_{d}}$. Defined function symbols will be denoted with $\widehat{\cdot}$. Schematic variable symbols are variables of the type $\omega \rightarrow \iota$ used to construct sequences of variables, essentially a generalization of the standard concept of a variable. Given a schematic variable $x$ instantiated by a numeral $\alpha$ we get a variable of the $\iota$ sort $x(\alpha)$. 

{\em Formula schemata}, a generalization of formulas including defined predicate symbols, are defined inductively using the standard logical connectives from uninterpreted and defined predicate symbols. Analogously, we label symbols as defined predicate symbols with $\widehat{\cdot}$. A {\em schematic sequent} is a pair of two multisets of formula schemata $\Delta$, $\Pi$ denoted by $\Delta \vdash \Pi$. We will denote multisets of formula schemata by uppercase greek letters unless it causes confusion. 

Note that we extend the \LK-calculus~\cite{GTakeuti1975} (see appendix \ref{appendix}) to the \LKE-calculus~\cite{CDunchev2014} by adding an inference rule for the construction of defined predicate and function symbols and a set of convergent rewrite rules $\mathcal{E}$ (equational theory) to our interpretation. The rules of $\mathcal{E}$ take the following form $\widehat{f}(\bar{t}) \equiv E$, where $\bar{t}$ contains no defined symbols, and either $\widehat{f}$ is a function symbol of range $\iota$ and $E$ is a term or $\widehat{f}$ is a predicate symbol and E is a formula schema.

\begin{definition}[\LKE]
Let $\mathcal{E}$ be an equational theory. \LKE\ is an extension of \LK\ by the
$\mathcal{E}$ inference rule
\begin{small}\begin{prooftree}
\AxiomC{$S(t)$}
\RightLabel{$\mathcal{E}$}
\UnaryInfC{$S(t')$}
\end{prooftree}\end{small}
where the term or formula schema $t$ in the sequent $S$ is replaced by a term or formula schema $t'$ for 
$\mathcal{E}\models t\equiv t'$.
\end{definition}

\begin{example}
Iterated version of  $\lor$ and $\land$ ( the defined predicates are abbreviated 
as $\bigvee$ and $\bigwedge$ ) can be defined using the following equational theory: 
\begin{align*}
\bigvee_{i=0}^0P(i)\equiv  P(0), & &
\bigvee_{i=0}^{s(y)}P(i)\equiv  \bigvee_{i=0}^{y}P(i)\lor P(s(y)), \\
\bigwedge_{i=0}^0P(i)\equiv  P(0) & & \bigwedge_{i=0}^{s(y)}P(i)\equiv \bigwedge_{i=0}^{y}P(i)\land P(s(y)).
\end{align*}

\end{example}

\subsection{ The \LKS-calculus and Proof Schemata}

Schematic proofs are a finite ordered list of {\em proof schema components} that can 
interact with each other. This interaction is defined using so-called 
\emph{links}, a 0-ary inference rule we add to \LKE-calculus: Let 
$S(k,\bar{x})$ be a sequent where $\bar{x}$ is a
vector of schematic variables.  By $S(k,\bar{t})$ we denote $S(k,\bar{x})$ where $\bar{x}$ is replaced by $\bar{t}$, respectively, and $\bar{t}$ is a vector of terms of appropriate type. 
Furthermore, we assume a countably infinite set $\mathcal{B}$ of \emph{proof symbols} denoted by $\varphi ,\psi ,\varphi_i,\psi_j$. The expression 
\begin{prooftree}
\AxiomC{$(\varphi,k,\bar{t})$}
\dashedLine
\UnaryInfC{$S(k,\bar{t})$}
\end{prooftree}
is called a link with the intended meaning that there is a proof called $\varphi$ with the
end-sequent $S(k,\bar{x})$. Let $k$ be a numeric expression, then $\mathcal{V}(k)$ is the set of parameters in $k$.  We refer to a link as an \emph{$E$-link} if $\mathcal{V}(k) \subseteq E$. Note that in this work $E = \left\lbrace n \right\rbrace$ or $E = \emptyset$. 

\begin{definition}[\LKS]
\LKS\ is an extension of \LKE, where links may appear at the leaves of a proof. 
\end{definition}


\begin{definition}[Proof Schema Component]
Let $\psi\in \mathcal{B}$ and $n\in \mathcal{N}$. A \emph{proof schema 
component} $\mathbf{C}$  is a triple $(\psi,\pi ,\nu(k))$  where $\pi$ is an 
\LKS-proof only containing $\emptyset$-links and $\nu(k)$ 
is an \LKS-proof containing $\left\lbrace n\right\rbrace $-links. The end-sequents of the proofs are  $S(0,\bar{x})$ and $S(k,\bar{x})$,
respectively. Given a proof schema 
component $\mathbf{C}=(\psi,\pi ,\nu(k))$ we define $\mathbf{C}.1 = \psi$, 
$\mathbf{C}.2 = \pi$, 
and  $\mathbf{C}.3 = \nu(k)$. 
\end{definition}

If $\nu(k)$ of a proof schema component $(\psi,\pi ,\nu(k))$  contains a link to $\psi$ it will be referred to as {\em cyclic}, otherwise it is {\em acyclic}. 

\begin{definition}[Proof Schema \cite{CDunchev2014}]
Let $\mathbf{C}_1,\cdots, \mathbf{C}_m$ be proof schema components such that 
$\mathbf{C}_i.1$ is distinct for $1\leq i\leq m$ and $n\in \mathcal{N}$. Let the  end sequents of $\mathbf{C}_1$ be $S(0,\bar{x})$ and $S(k,\bar{x})$. We define  $\Psi 
= \left\langle \mathbf{C}_1, \cdots, \mathbf{C}_m \right\rangle $ as a 
{\em proof schema} if 
$\mathbf{C}_i.3$ only contains $\left\lbrace n \right\rbrace $-links to $\mathbf{C}_i.1$ or 
$\mathbf{C}_j.1$ for $1\leq i < j \leq m$. The $\left\lbrace n \right\rbrace $-links are of the 
following form: 
\begin{center}
\begin{minipage}{.45\textwidth}
\begin{prooftree}
\AxiomC{$(\mathbf{C}_i.1, k',\bar{a})$}
\dashedLine
\UnaryInfC{$S'(k',\bar{a})$}
\end{prooftree}
\end{minipage}
\begin{minipage}{.45\textwidth}
\begin{prooftree}
\AxiomC{$(\mathbf{C}_j.1, t,\bar{b})$}
\dashedLine
\UnaryInfC{$S''(t,\bar{b})$}
\end{prooftree}
\end{minipage}
\end{center}
where $t$ is a numeric expression such that $\mathcal{V}(t)\subseteq \left\lbrace n \right\rbrace$, $k'$ is a sub-term of $k$, and $\bar{a}$ and 
$\bar{b}$ are vectors of terms from the appropriate sort. $S'(k',\bar{a})$ 
and $S''(t,\bar{b})$ are the end sequents of components $\mathbf{C}_i$ and 
$\mathbf{C}_j$ respectively.  We call $S(k,\bar{x})$ the end sequent 
of $\Psi$ and assume an identification between the formula occurrences in the end 
sequents of the proof schema components so that we can speak of occurrences in the 
end sequent of $\Psi$. The class of all proof schemata will be denoted by $\fat{\Upsilon}$. 
\end{definition}

For any proof schema $\Phi\in\fat{\Upsilon}$, such that $\Phi= \left\langle 
\mathbf{C}_1, \cdots, \mathbf{C}_m \right\rangle $ we define $|\Phi|= m$ and 
$\Phi.i = \mathbf{C}_i$ for $1\leq i \leq m$. Note that instead of using {\em proof schema pair}~\cite{CDunchev2014,ALeitsch2017} to define proof schemata we use proof schema components. The only difference is that 
proof schema components make the name explicit. All results concerning proof schemata built from proof schema pairs carry over for our above definition. 


\begin{example}
\label{example.1}
Let us consider the proof schema $\Phi = \left\langle (\varphi,\pi,\nu (k)) \right\rangle$. The proof schema uses one defined function symbol $\widehat{S}(\cdot)$ which is used to convert terms of the $\omega$ sort to the $\iota$ sort, that is
\[\mathcal{E} = \left\lbrace \widehat{S}(k+1) = f(\widehat{S}(k)) \ ; \ \widehat{S}(0) = 0 \ ; \ k+f(l)=f(k+l)\right\rbrace. \]
We abbreviate the context as $\Delta = \lbrace P(\alpha+0) ,\forall x.P(x)\to P(f(x)) \rbrace.$
The proofs $\pi$ and $\nu (k)$ are as follows: 
\begin{flushleft}
\begin{minipage}{.1\textwidth}
$\pi =$
\end{minipage}
\begin{minipage}{.85\textwidth}
\begin{tiny}
\begin{prooftree}
\AxiomC{$P(\alpha+0) \vdash P(\alpha +0)$}
\RightLabel{$w\colon l$}
\UnaryInfC{$\Delta \vdash P(\alpha +0)$}
\RightLabel{$\mathcal{E}$}
\UnaryInfC{$\Delta \vdash P(\alpha +\widehat{S}(0))$}
\end{prooftree}
\end{tiny}
\end{minipage}
\end{flushleft}
\begin{flushleft}
\begin{minipage}{.1\textwidth}
$\nu (k) =$
\end{minipage}
\begin{minipage}{.85\textwidth}
\begin{tiny}
\begin{prooftree}
\AxiomC{$(\varphi ,n,\alpha)$}
\dashedLine
\UnaryInfC{$\Delta \vdash P(\alpha + \widehat{S}(n) )$}
\AxiomC{$P(f(\alpha + \widehat{S}(n))) \vdash P(f(\alpha + \widehat{S}(n)))$}
\RightLabel{$\to\colon l$}
\BinaryInfC{$\Delta, P(\alpha + \widehat{S}(n) )\to P(f(\alpha + \widehat{S}(n) )) 
\vdash P(f(\alpha + \widehat{S}(n) ))$}
\RightLabel{$\forall\colon l$}
\UnaryInfC{$\Delta, \forall x.P(x)\to P(f(x)) \vdash P(f(\alpha + \widehat{S}(n) ))$}
\RightLabel{$\mathcal{E}$}
\UnaryInfC{$\Delta, \forall x.P(x)\to P(f(x)) \vdash  P(\alpha + f(\widehat{S}(n)))$}
\RightLabel{$\mathcal{E}$}
\UnaryInfC{$\Delta ,\forall x.P(x)\to P(f(x)) \vdash P(\alpha + \widehat{S}(n+1))$}
\RightLabel{$c\colon l$}
\UnaryInfC{$\Delta\vdash P(\alpha + \widehat{S}(n+1))$}
\end{prooftree}
\end{tiny}
\end{minipage}
\end{flushleft}
Note that $\pi$ contains no links, while $\nu (k)$ contains a single $\{ n\}$-link.
\end{example}

\section{Evaluation and Interpretation}
\label{sec:evalInter}

Proof schemata are an alternative formulation of induction. In~\cite{ALeitsch2017}, it is shown that proof schemata are equivalent to a fragment of the induction arguments formalizable in Peano arithmetic, i.e.\ the so called {\em $k$-simple induction}. More specifically, $k$-simple induction limits the number of inductive eigenvariables\footnote{Inductive eigenvariables are eigenvariables occurring in the context of an induction inference rule.} to one. In previous work~\cite{CDunchev2014,ALeitsch2017}, $\mathbf{LKE}$ was extended by the following induction rule instead of links: 
\begin{prooftree}
\AxiomC{$F(k),\Gamma \vdash \Delta, F(s(k))$}
\RightLabel{$\mathbf{IND}$}
\UnaryInfC{$F(0),\Gamma \vdash \Delta, F(t)$}
\end{prooftree}
where $t$ is an arbitrary term of the numeric sort. The result is the calculus $\mathbf{LKIE}$. To enforce $k$-simplicity we add the following constraint: let $\psi$ be an $\mathbf{LKIE}$-proof such that for any induction inference in $\psi$, $V(t)\subseteq \lbrace k\rbrace$ for some $k$.  In~\cite{CDunchev2014,ALeitsch2017}, the authors show that the following two proposition hold, and thus define the relationship between $k$-simple $\mathbf{LKIE}$-proofs and proof schemata. Given that our calculus can be used to construct proof schemata, the relationship can be trivially extended to proofs resulting from our calculus.
\begin{proposition}
Let $\Psi$ be a proof schema with end-sequent $\mathcal{S}$. Then there exists a $k$-simple $\mathbf{LKIE}$-proof of $\mathcal{S}$
\end{proposition}
\begin{proof}
See Proposition 3.13 of~\cite{ALeitsch2017}.
\end{proof}

\begin{proposition}
Let $\pi$ be a $k$-simple $\mathbf{LKIE}$-proof of $\mathcal{S}$. Then there exists a proof schema with end-sequent $\mathcal{S}$.
\end{proposition}
\begin{proof}
See Proposition 3.15 of~\cite{ALeitsch2017}.
\end{proof}

Unlike the induction proofs of the $\mathbf{LKIE}$-calculus, proof schemata have a recursive structure and thus require an evaluation (``unrolling''), similar to primitive recursion functions. When we instantiate the free parameter, the following evaluation procedure suffices. 
\begin{definition}[Evaluation of proof schema~\cite{CDunchev2014}]
\label{def:eval}
We define the rewrite rules for links 
\begin{flushleft}
\begin{minipage}{.45\textwidth}
\begin{small}\begin{prooftree}
\AxiomC{$(\varphi ,0,\bar{t})$}
\dashedLine
\RightLabel{$\Rightarrow\pi $}
\UnaryInfC{$S(0,\bar{t})$}
\end{prooftree}\end{small}
\end{minipage}
\begin{minipage}{.45\textwidth}
\begin{small}\begin{prooftree}
\AxiomC{$(\varphi ,k,\bar{t})$}
\dashedLine
\RightLabel{$\Rightarrow\nu (k)$}
\UnaryInfC{$S(k,\bar{t})$}
\end{prooftree}\end{small}
\end{minipage}
\end{flushleft}
for all proof schema components $\mathbf{C}=(\varphi ,\pi ,\nu (k))$. Furthermore, for $\alpha\in \mathbb{N}$, we define $\mathbf{C}\left[ k \setminus \alpha \right] \downarrow$ as a normal form of the link
\begin{small}\begin{prooftree}
\AxiomC{$(\varphi ,\alpha,\bar{t})$}
\dashedLine
\UnaryInfC{$S(\alpha,\bar{t})$}
\end{prooftree}\end{small}
 under the above rewrite system extended by the rewrite rules for defined function 
and predicate symbols, i.e. the equational theory $\mathcal{E}$. Also, for a proof schema $\Phi =\langle\mathbf{C}_1,\ldots 
,\mathbf{C}_m\rangle$, we define $\Phi\left[ n \setminus \alpha \right] \downarrow= \mathbf{C}_{1}\left[ k \setminus \alpha \right] \downarrow$.
\end{definition}

\begin{example}
Let $\Phi$ be the proof schema of example \ref{example.1} and $\Delta$ defined equivalently. For $1\in\mathbb{N}$ we can write down $\Phi\left[ n\setminus 1 \right] \downarrow$ as follows:
\begin{center}
\begin{tiny}
\begin{prooftree}
\AxiomC{$P(f(\alpha + \widehat{S}(0))) \vdash P(f(\alpha + \widehat{S}(0)))$}
\AxiomC{$P(\alpha + 0)\vdash P(\alpha + 0)$}
\RightLabel{$w\colon l$}
\UnaryInfC{$\Delta \vdash P(\alpha + 0)$}
\RightLabel{$\mathcal{E}$}
\UnaryInfC{$\Delta \vdash P(\alpha + \widehat{S}(0))$}
\RightLabel{$\to\colon l$}
\BinaryInfC{$\Delta ,P(\alpha + \widehat{S}(0) )\to P(f(\alpha + \widehat{S}(0) )) 
\vdash P(f(\alpha + \widehat{S}(0)))$}
\RightLabel{$\forall\colon l$}
\UnaryInfC{$\Delta ,\forall x.P(x)\to P(f(x)) \vdash P(f(\alpha + \widehat{S}(0) ))$}
\RightLabel{$c\colon l$}
\UnaryInfC{$\Delta \vdash  P(f(\alpha + \widehat{S}(0) ))$}
\RightLabel{$\mathcal{E}$}
\UnaryInfC{$\Delta \vdash P(\alpha + f(\widehat{S}(0)))$}
\RightLabel{$\mathcal{E}$}
\UnaryInfC{$\Delta \vdash P(\alpha + \widehat{S}(0+1))$}
\RightLabel{$\mathcal{E}$}
\UnaryInfC{$\Delta \vdash P(\alpha + \widehat{S}(1))$}
\end{prooftree}
\end{tiny}
\end{center}
\end{example}

The described evaluation procedure essentially defines a rewrite system for proof schemata with the following property.

\begin{lemma}
The rewrite system for links is strongly normalizing, and for a proof schema $\Phi$ and $\alpha\in \mathcal{N}$, $\Phi\left[ n \setminus \alpha \right] \downarrow$ is an $\mathbf{LK}$-proof. 
\end{lemma}

\begin{proof}
See Lemma 3.10 of~\cite{ALeitsch2017}. 
\end{proof}

\begin{proposition}[Soundness \cite{CDunchev2014}]
\label{prop:sound}
Let $\Phi=$ $\langle\mathbf{C}_1,\ldots ,\mathbf{C}_m\rangle$ be a proof schema with end-sequent $S(n,\bar{x})$ and let $\alpha\in \mathcal{N}$. Then $\Phi\left[ n \setminus \alpha \right] \downarrow$ is an \LK-proof of $S(n,\bar{x})$.
\end{proposition}

Essentially, Proposition~\ref{prop:sound} states that $\mathbf{C}_1\left[ k \setminus \alpha \right] \downarrow$ is an \LK-proof of the end-sequent $S(n,\bar{x})\left[ k \setminus \alpha \right]\downarrow$ where by $\downarrow$ we refer to normalization of the defined symbols in $S(n,\bar{x})$. 

\section{The $\mathcal{S}\mathit{i}\mathbf{LK}$-calculus}
\label{sec:calculus}

The $\mathcal{S}\mathit{i}\mathbf{LK}$-calculus ($\mathcal{S}$chematic $\mathit{i}$nduction $\mathbf{LK}$-calculus, see Table \ref{tab:Calculus} \& \ref{tab:Calculus2}) allows one to build a proof schema component-wise. We call the set of expressions in between two $|$ a {\em component group}. Note that, unlike proof schemata we do not need proof symbols nor ordering because it is implied by the construction. Each component group consists of a multiset of {\em component pairs} which are pairs of $\mathbf{LKS}$-sequents. A set of component groups is referred to as a {\em component collection}.  Even though all auxiliary components (or component groups) are shifted to the left, we do not intend any ordering, i.e.\ writing, for instance,  $\SchCom{\top}{A\vdash A}$ to the right of $\fat{\Pi}$ in $Ax_1:r$ does not change the rule. 

\begin{table}
\caption{The basic inference rules of the $\mathcal{S}\mathit{i}\mathbf{LK}$-calculus. }
\label{tab:Calculus}
\begin{tabular}{|c|}
\hline

\begin{minipage}{.45\textwidth} 
\begin{tiny}

\begin{prooftree}
\AxiomC{$ 
\fat{\Pi} 
$}
\RightLabel{$Ax_1:r$}
\UnaryInfC{$ 
\SchCom{\top}{A\vdash A\ } 
\ComSep{}{} 
\fat{\Pi} 
$}
\end{prooftree}
\end{tiny}

\end{minipage}
\begin{minipage}{.45\textwidth} 
\begin{tiny}

\begin{prooftree}
\AxiomC{$ 
\mathbb{\Gamma} 
\ComSep{}{} 
\fat{\Pi} 
$}
\RightLabel{$Ax_2:r$}
\UnaryInfC{$ 
\SchCom{\top}{A\vdash A\ }, 
\mathbb{\Gamma} 
\ComSep{}{} 
\fat{\Pi} 
$}
\end{prooftree}
\end{tiny}

\end{minipage} \\

\begin{minipage}{.45\textwidth} 
\begin{tiny}

\begin{prooftree}
\AxiomC{$ 
\SchCom{\top}{\fat{[}\ \fat{\pi}\ \fat{]}}, 
\mathbb{\Gamma} 
\ComSep{}{} 
\fat{\Pi} 
$}
\RightLabel{$Ax:l$}
\UnaryInfC{$ 
\SchCom{A\vdash^{f(n)} A}{\fat{[}\ \fat{\pi}\ \fat{]}}, 
\mathbb{\Gamma} 
\ComSep{}{} 
\fat{\Pi} 
$}
\end{prooftree}
\end{tiny}

\end{minipage}
\begin{minipage}{.45\textwidth} 
\begin{tiny}

\begin{prooftree}
\AxiomC{$ 
\SchCom{\top}{\fat{\pi}}, 
\SchCom{\top}{\fat{\pi}}, 
\mathbb{\Gamma} 
\ComSep{}{} 
\fat{\Pi} 
$}
\RightLabel{$c_{c}:r$}
\UnaryInfC{$ 
\SchCom{\top}{\fat{\pi}}, 
\mathbb{\Gamma} 
\ComSep{}{} 
\fat{\Pi} 
$}
\end{prooftree}
\end{tiny}

\end{minipage} \\

\begin{minipage}{.5\textwidth} 
\begin{tiny}

\begin{prooftree}
\AxiomC{$ 
\SchCom{\fat{\nu}}{\fat{[}\ \fat{\pi}\ \fat{]}}, 
\SchCom{\fat{\nu}}{\fat{[}\ \fat{\pi}\ \fat{]}}, 
\mathbb{\Gamma} 
\ComSep{}{} 
\fat{\Pi} 
$}
\RightLabel{$c_{c}:l$}
\UnaryInfC{$ 
\SchCom{\fat{\nu}}{\fat{[}\ \fat{\pi}\ \fat{]}}, 
\mathbb{\Gamma} 
\ComSep{}{} 
\fat{\Pi} 
$}
\end{prooftree}
\end{tiny}

\end{minipage}
\begin{minipage}{.4\textwidth} 
\begin{tiny}

\begin{prooftree}
\AxiomC{$ 
\SchCom{\fat{\nu}}{\fat{[}\fat{\pi} \fat{]}}, 
\mathbb{\Gamma} 
\ComSep{}{} 
\fat{\Pi} 
$}
\RightLabel{$br$}
\UnaryInfC{$ 
\SchCom{\top}{\fat{[}\fat{\pi} \fat{]}}, 
\mathbb{\Gamma'} 
\ComSep{}{} 
\fat{\Pi} 
$}
\end{prooftree}
\end{tiny}

\end{minipage} \\

\begin{minipage}{.45\textwidth} 
\begin{tiny}

\begin{prooftree}
\AxiomC{$ 
\SchCom{\top}{ \fat{\nu} }, 
\mathbb{\Gamma} 
\ComSep{}{} 
\fat{\Pi} 
$}
\RightLabel{$cl_{bc}$}
\UnaryInfC{$ 
\SchCom{\top}{\fat{[}\ \fat{\nu} \ \fat{]}}, 
\mathbb{\Gamma} 
\ComSep{}{} 
\fat{\Pi} 
$}
\end{prooftree}
\end{tiny}

\end{minipage}
\begin{minipage}{.45\textwidth} 
\begin{tiny}

\begin{prooftree}
\AxiomC{$ 
\SchCom{\top}{\fat{[}\ \fat{\nu} \ \fat{]}}
\ComSep{}{} 
\fat{\Pi} 
$}
\RightLabel{$cl_{LKE}$}
\UnaryInfC{$ 
\SchCom{\fat{[}\ \fat{]}}{\fat{[}\ \fat{\nu} \ \fat{]}}
\ComSep{}{} 
\fat{\Pi} 
$}
\end{prooftree}

\end{tiny}

\end{minipage} \\

\begin{minipage}{.9\textwidth} 
\begin{tiny}

\begin{prooftree}
\AxiomC{$ 
\SchCom{(\Pi\vdash^{\alpha} \Delta)\left[ n \setminus \alpha \right]}{\fat{[}\ (\Pi\vdash \Delta)\left[ n \setminus 0 \right] \ \fat{]}} 
\ComSep{}{} 
\fat{\Pi} 
$}
\RightLabel{$cl_{sc}$}
\UnaryInfC{$ 
\SchCom{\fat{[}\ (\Pi\vdash \Delta)\left[ n \setminus \alpha \right] \ \fat{]}}{\fat{[}\ (\Pi\vdash \Delta)\left[ n \setminus 0 \right] \ \fat{]}} 
\ComSep{}{} 
\fat{\Pi} 
$}
\end{prooftree}
\end{tiny}

\end{minipage} \\

\begin{minipage}{.5\textwidth} 
\begin{tiny}

\begin{prooftree}
\AxiomC{$ 
\SchCom{\fat{\nu}}{\fat{[}\ \fat{\pi} \ \fat{]}},\SchCom{\fat{\mu}}{\fat{[}\ \fat{\pi} \ \fat{]}}, 
\mathbb{\Gamma}
\ComSep{}{} 
\fat{\Pi} 
$}
\RightLabel{$\rho_2^{sc}$}
\UnaryInfC{$ 
\SchCom{\fat{\nu}'}{\fat{[}\ \fat{\pi} \ \fat{]}}, 
\mathbb{\Gamma} 
\ComSep{}{} 
\fat{\Pi} 
$}
\end{prooftree}
\end{tiny}

\end{minipage}
\begin{minipage}{.4\textwidth} 
\begin{tiny}

\begin{prooftree}
\AxiomC{$ 
\SchCom{\fat{\nu}}{\fat{[}\ \fat{\pi} \ \fat{]}}, 
\mathbb{\Gamma} 
\ComSep{}{} 
\fat{\Pi} 
$}
\RightLabel{$\rho_1^{sc}$}
\UnaryInfC{$ 
\SchCom{\fat{\nu}'}{\fat{[}\ \fat{\pi} \ \fat{]}}, 
\mathbb{\Gamma} 
\ComSep{}{} 
\fat{\Pi} 
$}
\end{prooftree}
\end{tiny}

\end{minipage} \\

\begin{minipage}{.5\textwidth} 
\begin{tiny}

\begin{prooftree}
\AxiomC{$ 
\SchCom{\top}{\fat{\pi}}, 
\SchCom{\top}{\fat{\eta}}, 
\mathbb{\Gamma} 
\ComSep{}{} 
\fat{\Pi} 
$}
\RightLabel{$\rho_2^{bc}$}
\UnaryInfC{$ 
\SchCom{\top}{\fat{\pi}'}, 
\mathbb{\Gamma} 
\ComSep{}{} 
\fat{\Pi} 
$}
\end{prooftree}
\end{tiny}

\end{minipage}
\begin{minipage}{.4\textwidth} 
\begin{tiny}

\begin{prooftree}
\AxiomC{$ 
\SchCom{\top}{\fat{\pi}}, 
\mathbb{\Gamma} 
\ComSep{}{} 
\fat{\Pi} 
$}
\RightLabel{$\rho_1^{bc}$}
\UnaryInfC{$ 
\SchCom{\top}{\fat{\pi}'}, 
\mathbb{\Gamma} 
\ComSep{}{} 
\fat{\Pi} 
$}
\end{prooftree}
\end{tiny}

\end{minipage} \\

\hline
\end{tabular}
\end{table}

To enforce correct construction of proof schema components we introduce a {\em closure} mechanism similar to {\em focusing}~\cite{JAndreoli1992}. Let us consider a component pair $\mathbf{C}=\SchCom{\fat{\nu}}{\fat{\pi}}$ where $\fat{\nu}$ is a sequent, a sequent in square brackets, or $\top$ and $\fat{\pi}$ is a sequent or a sequent in square brackets. The left side $\fat{\nu}$ is the stepcase and the right side $\fat{\pi}$ is the basecase. The configuration $\fat{\nu} =\top$ means that we are still allowed to apply rules to the basecase. If $\fat{\pi}$ is closed, i.e.\ $\fat{\pi}$ is of the form $\left[\ \Delta\vdash\Pi \ \right]$ for an arbitrary sequent $\Delta\vdash\Pi$, we have closed the basecase (using inference rule $cl_{bc}$) and essentially fixed its end-sequent. Therefore, $\fat{\nu}$ is always equal to $\top$ as long as the basecase is not closed. If $\fat{\pi}$ is closed we are allowed to apply rules to the stepcase. This fixing of the end-sequent essentially fixes the sequent we are allowed to introduce using the inference rule $\circlearrowright$. 

Apart from schematic proofs, simple \LKE-proofs can be constructed by keeping the stepcase equal to $\top$. If we instead intend to proof a scheme we have to work on the stepcase. If we consider a non-redundant basecase\footnote{In general, it is allowed to construct any number of components that are not necessary for the proof construction. Here, we want to consider a basecase of a relevant component.}, its sequent $(\Pi\vdash\Delta )[n\setminus 0]$ characterizes already the sequent we have to construct by the stepcase, i.e.\ $(\Pi\vdash^\alpha\Delta )[n\setminus \alpha]$ where $\alpha$ depends on the applications of $\circlearrowright\ $- or $\curvearrowright\ $-inferences. Note that $f(n)$, $g(n)$, and $h(n)$ are intended to be arbitrary primitive recursive functions and may be introduced as an extension of the equational theory. 

When a component group contains a single component pair and the endsequents of the basecase and stepcase are the same modulo the substitution of the free parameter we can close the component group using $cl_{sc}$. Alternatively, we can close a group by applying $cl_{LKE}$ if the stepcase is equal to $\top$.  We refer to such a group as a {\em closed group} and any group which is not closed is referred to as an {\em open group}. As a convention, inference rules can only be applied to open groups. Concerning  $\curvearrowright$, it may be the case that the closed group whose end sequent we use to introduce a link has free variables other than the free parameter. We assume correspondence between the free variables of the closed group and the introduced, meaning that in a call 
\begin{align*}
\SchCom{(\Lambda \vdash^{f(n)} \Gamma)\left[ n \setminus g(n) \right]\left[ \bar{x}\setminus\bar{t}\right]}{\fat{[}\ \fat{\pi}\ \fat{]}} 
\end{align*}
of a component with free variables $\bar{x}$ all occurrences of $\bar{x}$ in the proof of 
\begin{align*}
\SchCom{\fat{[}\ (\Lambda  \vdash\Gamma)\left[ n \setminus h(n) \right]\ \fat{]}}{\fat{[}\ \fat{\delta}\ \fat{]}} 
\end{align*}
are replaced with $\bar{t}$.

An $\mathcal{S}\mathit{i}\mathbf{LK}$-derivation is a sequence of $\mathcal{S}\mathit{i}\mathbf{LK}$ inferences rules ending in a component collection with at least one open component group. An $\mathcal{S}\mathit{i}\mathbf{LK}$-proof ends in a component collection where all components are closed. As we shall show, not every derivation can be extended into a proof. 

\begin{table}
\caption{The linking rules of the $\mathcal{S}\mathit{i}\mathbf{LK}$-calculus. }
\label{tab:Calculus2}
\begin{tabular}{|c|}
\hline\\
\begin{minipage}{0.9\textwidth}
\begin{tiny}
\begin{prooftree}
\AxiomC{$ 
\SchCom{\top}{\fat{[}\ (\Pi \vdash \Delta)\left[ n \setminus 0 \right] \ \fat{]}}, 
\mathbb{\Gamma} 
\ComSep{}{} 
\fat{\Pi} 
$}
\RightLabel{$\circlearrowright$}
\UnaryInfC{$ 
\SchCom{ 
{(\Pi \vdash^{(n+1)} \Delta)\left[ n \setminus n \right]\left[ \bar{x}\setminus\bar{t}\right]}
}{\fat{[}\ (\Pi\vdash \Delta)\left[ n \setminus 0 \right] \ \fat{]}}, 
\mathbb{\Gamma} 
\ComSep{}{} 
\fat{\Pi} 
$}
\end{prooftree}
\end{tiny}

\end{minipage} \\

\\

{\footnotesize where $\bar{x}$ is the vector of all free variables of $(\Pi\vdash\Delta )$ and $\bar{t}$ is an} \\ 
{\footnotesize arbitrary vector of terms which has the same length as $\bar{x}$.} \\

\begin{minipage}{0.9\textwidth}
\begin{tiny}

\begin{prooftree}
\AxiomC{$ 
\SchCom{\top}{\fat{[}\ \fat{\pi}\ \fat{]}}, 
\mathbb{\Gamma} 
\ComSep{}{} 
\fat{\Delta}\ComSep{}{} \SchCom{\fat{[}\ (\Lambda  \vdash\Gamma)\left[ n \setminus h(n) \right]\ \fat{]}}{\fat{[}\ \fat{\delta}\ \fat{]}}
\ComSep{}{} 
\fat{\Pi} 
$}
\RightLabel{$\curvearrowright$}
\UnaryInfC{$ 
\SchCom{(\Lambda \vdash^{f(n)} \Gamma)\left[ n \setminus g(n) \right]\left[ \bar{y}\setminus\bar{t}\right]}{\fat{[}\ \fat{\pi}\ \fat{]}}, 
\mathbb{\Gamma} 
\ComSep{}{} 
 \fat{\Pi'} 
$}
\end{prooftree}
\end{tiny}

\end{minipage} \\

\\

{\footnotesize where $\bar{y}$ is the vector of all free variables of $(\Lambda\vdash\Gamma )$ and $\bar{t}$ is an}\\ 
{\footnotesize arbitrary vector of terms which has the same length as $\bar{y}$. Also, $g$}\\ {\footnotesize and $h$ are arbitrary primitive recursive functions.
}
\\
\hline
\end{tabular}
\end{table}

We also consider a special case of $\mathcal{S}\mathit{i}\mathbf{LK}$-derivations (proofs) which we refer to as {\em pre-proof schema normal form}. A  $\mathcal{S}\mathit{i}\mathbf{LK}$-derivation (proof) is in pre-proof schema normal form if for every application of $Ax_{1}:r$ the context $\fat{\Pi}$ is a $\mathcal{S}\mathit{i}\mathbf{LK}$-proof. This enforces  a stricter order on the construction of components than it is already enforced by the use of the $\curvearrowright$-inference which can be used to construct proof schemata.

Let $\mathcal{I}$ be the customary evaluation function of sequents, i.e.\ $\mathcal{I}(\Delta\vdash\Gamma )\equiv\bigwedge_{F\in\Delta}F\to\bigvee_{F\in\Gamma}F$ for an \LKE-sequent $\Delta\vdash\Gamma$ and assume an $\mathcal{S}\mathit{i}\mathbf{LK}$-proof ending in the component collection
\begin{align*}
\mathbf{C} \equiv \SchCom{\fat{[}\ \fat{\nu_0}\ \fat{]}\ }{\fat{[}\ \fat{\pi_0} \ \fat{]}}
\ComSep{}{} 
\cdots 
\ComSep{}{}
\SchCom{\fat{[}\ \fat{\nu_m}\ \fat{]}\ }{\fat{[}\ \fat{\pi_m} \ \fat{]}}
\end{align*}
such that $\SchCom{\fat{[}\ \fat{\nu_0}\ \fat{]}\ }{\fat{[}\ \fat{\pi_0}\fat{]}}$ is the  last component group closed in the proof of $C$ (In the following we will refer to this component as the {\em leading component}). We extend the evaluation function to the schematic case and define the {\em evaluation function} of a closed component collection similar to \cite{GGentzen1969} by
\begin{align*}
&\mathcal{I}_{\mathcal{S}\mathit{i}\mathbf{LK}}(\mathbf{C})\equiv \mathcal{I}(\fat{\pi_0}),
\end{align*}
if $\nu_0\equiv\top$ and 
\begin{align*}
&\mathcal{I}_{\mathcal{S}\mathit{i}\mathbf{LK}}(\mathbf{C})\equiv \\ 
&\bigwedge_{i=0}^{m} \mathcal{I}(\fat{\pi_i})\wedge
\forall. x\Big(\bigwedge_{i=0}^{m}\big(\mathcal{I}(\fat{\nu_{i}}\left[ n\setminus x \right]) \rightarrow \mathcal{I}(\fat{\nu_{i}}\left[ n\setminus (x+1) \right])\big)\Big)
\rightarrow \forall x. (\mathcal{I}(\fat{\nu_{0}}\left[ n\setminus x \right]) ,
\end{align*}
otherwise. Implicitly, the closure rules imply an order. In general, all closed component groups are considered lower in the implied ordering than open component groups. Essentially, the ordering comes from the $\curvearrowright$-rule which can only be applied if the auxiliary component is closed. For example, a component may be forced to be closed last, and thus, would be consider the top of the implied ordering. 

We use the following denotations for construction of our inference rules. Context variables within schematic sequents will be denoted by uppercase greek letters $\Delta, \Pi$, etc. Context variables within component groups will be denoted by blackboard bold uppercase greek letters $\mathbb{\Delta}, \mathbb{\Pi}$, etc. Context variables within the component collection will be denoted by fat bold uppercase greek letters, $\fat{\Delta}, \fat{\Pi}$, etc. We use bold lowercase greek letters to denote schematic sequents, $\fat{\delta},\fat{\pi}$, etc. The inference rules $\rho_1^{sc}$, $\rho_2^{sc}$, $\rho_1^{bc}$, $\rho_2^{bc}$ apply an $\mathbf{LKE}$ inference rule $\rho$ to the auxiliary sequents to get the main sequent. By the subscript we denote the arity of the inference rule. For example, $(\forall:l)_1^{sc}$ applies the universal quantifier rule to the left side of the stepcase. And finally, we use the following abbreviations: 
\begin{align*}
&\mathbb{\Gamma'} \equiv \left(  \mathbb{\Gamma}, \SchCom{\fat{\nu}}{\fat{[}\fat{\pi} \fat{]}}\right) ,\text{ and} \\ 
&\fat{\Pi'} \equiv  \fat{\Delta}, \ \SchCom{\fat{[}\ (\Lambda \vdash\Gamma)\lbrace n \leftarrow \alpha\rbrace\ \fat{]}}{\fat{[}\ \fat{\delta}\ \fat{]}} \ComSep{}{}\fat{\Pi}.
\end{align*}

The following example illustrates the construction of a simple $\mathcal{S}\mathit{i}\mathbf{LK}$-proof. 
\begin{example}
\label{ex:silkProof}
For the construction of the following $\mathcal{S}\mathit{i}\mathbf{LK}$-proof we use the equational theory $\mathcal{E} \equiv \{ \widehat{f^0}(x)=x;\widehat{f^{s(n)}}(x)=f\widehat{f^n}(x)\}$ and the abbreviations 
\begin{align*}
&\Delta \equiv P(0),\forall x.P(x)\to P(f(x)) \text{ and }\fat{\pi} \equiv \Delta\vdash P(\widehat{f^0(0)}).
\end{align*}

\begin{center}
\begin{tiny}
\begin{prooftree}
\AxiomC{}
\RightLabel{$Ax_{1}:r$}
\UnaryInfC{$
\SchCom{\top}{P(0)\vdash P(0)} 
\ComSep{}{} 
$}
\RightLabel{$\mathcal{E}_{1}^{bc}$}

\UnaryInfC{$
\SchCom{\top}{P(0)\vdash P(\widehat{f^0}(0))} 
\ComSep{}{} 
$}
\RightLabel{$(w:l)_{1}^{bc}$}
\UnaryInfC{$
\SchCom{\top}{P(0),\forall x.P(x)\to P(f(x))\vdash P(\widehat{f^0}(0))} 
\ComSep{}{} 
$}
\RightLabel{$cl_{bc}$}
\UnaryInfC{$
\SchCom{\top}{\fat{[}\Delta\vdash P(\widehat{f^0}(0))\ \fat{]}} 
\ComSep{}{} 
$}
\RightLabel{$Ax:l$}
\UnaryInfC{$
\SchCom{P(f\widehat{f^n}(0))\vdash^{s(n)} P(f\widehat{f^n}(0))}{\fat{[}\ \fat{\pi} \ \fat{]}} 
\ComSep{}{} 
$}
\RightLabel{$br$}
\UnaryInfC{$
\SchCom{\top}{\fat{[}\ \fat{\pi} \ \fat{]}},  
\SchCom{P(f\widehat{f^n}(0))\vdash^{s(n)} P(f\widehat{f^n}(0))}{\fat{[}\ \fat{\pi} \ \fat{]}} 
\ComSep{}{} 
$}
\RightLabel{$\circlearrowright$}
\UnaryInfC{$
\SchCom{\Delta\vdash^{s(n)} P(\widehat{f^n}(0)) }{\fat{[}\ \fat{\pi} \ \fat{]}}, 
\SchCom{P(f\widehat{f^n}(0))\vdash^{s(n)} P(f\widehat{f^n}(0))}{\fat{[}\ \fat{\pi} \ \fat{]}} 
\ComSep{}{} 
$}
\RightLabel{$(\rightarrow:l)_{2}^{sc}$}
\UnaryInfC{$
\SchCom{\Delta,P(\widehat{f^n}(0))\to P(f\widehat{f^n}(0))\vdash^{s(n)} P(f\widehat{f^n}(0))}{\fat{[}\ \fat{\pi} \ \fat{]}} 
\ComSep{}{} 
$}
\RightLabel{$(\forall:l)_{1}^{sc}$}
\UnaryInfC{$
\SchCom{\Delta,\forall x.P(x)\to P(f(x))\vdash^{s(n)} P(\widehat{f^{s(n)}}(0))}{\fat{[}\ \fat{\pi} \ \fat{]}} 
\ComSep{}{} 
$}
\RightLabel{$(c:l)_{1}^{sc}$}
\UnaryInfC{$
\SchCom{\Delta\vdash^{s(n)} P(\widehat{f^{s(n)}}(0))}{\fat{[}\ \fat{\pi} \ \fat{]}} 
\ComSep{}{} 
$}
\RightLabel{$cl_{sc}$}

\UnaryInfC{$
\SchCom{\fat{[}\ \Delta \vdash P(\widehat{f^{s(n)}}(0)) \ \fat{]}}{\fat{[}\ \fat{\pi} \ \fat{]}} 
\ComSep{}{} 
$}
\end{prooftree}
\end{tiny}
\end{center}

By applying the evaluation function $\mathcal{I}_{\mathcal{S}\mathit{i}\mathbf{LK}}$ we get 
\begin{align*}
&\mathcal{I}_{\mathcal{S}\mathit{i}\mathbf{LK}}\left( \SchCom{\fat{[}\ \Delta \vdash P(\widehat{f^{s(n)}}(0)) \ \fat{]}}{\fat{[}\ \fat{\pi} \ \fat{]}} \right) \equiv \\
& \bigg( \big( \Delta'\to P(\widehat{f^0}(0))\big) \wedge \forall x. \Big( \big( \Delta'\to P(\widehat{f^x}(0))\big) \to 
\big( \Delta'\to P(\widehat{f^{x+1}}(0)) \big) \Big) \bigg) \to \\
& \big( \Delta'\to \forall n.P(\widehat{f^n}(0))\big)
\end{align*}
where $\Delta' \equiv P(0)\wedge \forall x.P(x)\to P(f(x))$.
Essentially, what we have proven with this $\mathcal{S}\mathit{i}\mathbf{LK}$-proof is the sequent  \[P(0),\forall x.P(x)\to P(f(x))\vdash \forall n.P(\widehat{f^n}(0)).\]
\end{example}

By extending the equational theory and by applying the $\curvearrowright$-inference we can easily strengthen the provable sequent of Example~\ref{ex:silkProof}. 

\begin{example}
Let $\mathcal{E}$ be the equational theory of Example~\ref{ex:silkProof}, $\fat{\Pi}$ the proof of Example~\ref{ex:silkProof} and 
\begin{align*}
\Delta \equiv &\; P(0),\forall x.P(x)\to P(f(x)), \\
\mathbb{\Gamma}\equiv &\; \SchCom{\fat{[}\ \Delta \vdash P(\widehat{f^{s(n)}}(0)) \ \fat{]}}{\fat{[}\ \Delta\vdash P(\widehat{f^0(0)}) \ \fat{]}}, \\ 
2\equiv &\; s(s(0)), \\
\fat{\pi}'\equiv &\; \Delta\vdash P(\widehat{f^{2^{0}}}(0)), \\
\mathcal{E}'\equiv &\; \mathcal{E}\cup\{ \widehat{f^{2^0}}(x)=f(x),\widehat{f^{2^{s(n)}}}(x)=\widehat{f^{2^n}}\widehat{f^2}(x)\}.
\end{align*}

\begin{tiny}
\begin{prooftree}
\AxiomC{$\fat{\Pi}$}
\RightLabel{$Ax_1:r$}

\UnaryInfC{$\SchCom{\top}{P(f(0))\vdash P(f(0))} 
\ComSep{}{}  \mathbb{\Gamma}
\ComSep{}{} 
$}
\RightLabel{$\mathcal{E}_{1}^{bc}$}

\UnaryInfC{$
\SchCom{\top}{P(f(0))\vdash P(\widehat{f^{2^0}}(0))} 
\ComSep{}{}  \mathbb{\Gamma}
\ComSep{}{} 
$}
\RightLabel{$Ax_{2}:r$}
\UnaryInfC{$
\SchCom{\top}{P(0)\vdash P(0)} 
,\SchCom{\top}{P(f(0))\vdash P(\widehat{f^{2^0}}(0))} 
\ComSep{}{}  \mathbb{\Gamma}
\ComSep{}{} $}
\RightLabel{$(\rightarrow:l)_{2}^{bc}$}
\UnaryInfC{$
\SchCom{\top}{P(0), P(0)\rightarrow P(f(0))\vdash P(\widehat{f^{2^0}}(0))} 
\ComSep{}{}  \mathbb{\Gamma}
\ComSep{}{} $}
\RightLabel{$(\forall:l)_{1}^{bc}$}
\UnaryInfC{$
\SchCom{\top}{P(0), \forall x (P(x)\rightarrow P(f(x)))\vdash P(\widehat{f^{2^0}}(0))} 
\ComSep{}{}  \mathbb{\Gamma}
\ComSep{}{} $}
\RightLabel{$cl_{bc}$}
\UnaryInfC{$
\SchCom{\top}{\fat{[}\ \Delta\vdash P(\widehat{f^{2^{0}}}(0))\ \fat{]}} 
\ComSep{}{}  \mathbb{\Gamma}
\ComSep{}{}
$}
\RightLabel{$\curvearrowright$}
\UnaryInfC{$
\SchCom{\Delta\vdash^{2^{s(n)}} P(\widehat{f^{2^{s(n)}}}(0))}{\fat{[}\ \fat{\pi'} \ \fat{]}}
\ComSep{}{}  \mathbb{\Gamma}
\ComSep{}{} 
$}
\RightLabel{$cl_{sc}$}
\UnaryInfC{$
\SchCom{\fat{\Big[}\Delta\vdash^{2^{s(n)}} P(\widehat{f^{2^{s(n)}}}(0))\fat{\Big]}}{\fat{[}\ \fat{\pi'} \ \fat{]}}
\ComSep{}{}  \mathbb{\Gamma}
\ComSep{}{} 
$}
\end{prooftree}
\end{tiny}
\end{example}
Notice that we were able to get a much stronger theorem without cuts or significantly extending the proof. Though, the instantiation of the second proof will be exponentially larger than an instantiation of the first proof for the same value of $n$, the second proof is only double the number of inferences. This is precisely the method one can use to formalize either Orevkov's proof~\cite{VOrevkov1991} or Boolos' proof~\cite{GBoolos1984}.

\subsection{From $\mathcal{S}\mathit{i}\mathbf{LK}$-Proof to Proof Schema}
\label{Sec:Translation}

It is possible to construct a proof schema from any $\mathcal{S}\mathit{i}\mathbf{LK}$-Proof, though it is much easier to perform the translation from $\mathcal{S}\mathit{i}\mathbf{LK}$-Proof in pre-proof schema normal form. We now show that every $\mathcal{S}\mathit{i}\mathbf{LK}$-Proof has a pre-proof schema normal form. 

\begin{lemma}
\label{lem:PPSNF}
Let $\Phi$ be a $\mathcal{S}\mathit{i}\mathbf{LK}$-Proof of a component collection $\mathbf{C}$. Then there exists a $\mathcal{S}\mathit{i}\mathbf{LK}$-Proof $\Phi'$  of $\mathbf{C}$ in pre-proof schema normal form.
\end{lemma}
\begin{proof}
We prove the statement by rearranging the application of the $\mathcal{S}\mathit{i}\mathbf{LK}$ rules. Let 
\begin{align*}
\mathbf{C} \equiv \SchCom{\fat{[}\ \fat{\nu_0}\ \fat{]}\ }{\fat{[}\ \fat{\pi_0} \ \fat{]}}
\ComSep{}{} 
\cdots 
\ComSep{}{}
\SchCom{\fat{[}\ \fat{\nu_m}\ \fat{]}\ }{\fat{[}\ \fat{\pi_m} \ \fat{]}}
\end{align*}
be the ending component collection of $\Phi$. We identify each component group $\mathcal{CG}_i=\SchCom{\fat{[}\ \fat{\nu_i}\ \fat{]}\ }{\fat{[}\ \fat{\pi_i} \ \fat{]}}$ with its ancestors in $\Phi$, i.e.\ all component groups that are connected via a $\mathcal{S}\mathit{i}\mathbf{LK}$ rule, exempting the closed component groups of the $\curvearrowright$ rule, to $\mathcal{CG}_i$. Afterwards, we find the component group $\mathcal{CG}_i$ which is closed first, i.e.\ reading top to bottom the component group to which $cl_{st}$ or $cl_{LKE}$ is applied first. Since there is no other component closed earlier there is no $\curvearrowright$ rule identified with $\mathcal{CG}_i$, thus we can consider all rules identified with $\mathcal{CG}_i$ to be independent. This implies that we can rearrange $\Phi$ such that all rules identified with $\mathcal{CG}_i$ are at the top\footnote{In general, the context is not empty. Since the rules, exempting the $\curvearrowright$, are independent from the context, we can always adjust the context.}. This part of the proof will not change any more. 
Now, we look again for the topmost $cl_{st}$ or $cl_{LKE}$ rule apart from the one we already considered. The corresponding component group $\mathcal{CG}_j$ and its identified rules contain only $\curvearrowright$ rules that link to components that are already rearranged and, hence, we can shift all rules identified with $\mathcal{CG}_j$ directly after the already rearranged ones. If we repeat this procedure, we end up with a proof in pre-proof schema normal form. 
\end{proof}

The important property of pre-proof schema normal form is that the construction of components is organized such that between any two closure rules is an $\LKS$-proof.  

\begin{theorem}
Let $\Phi$ be an $\mathcal{S}\mathit{i}\mathbf{LK}$-Proof of the component collection \[\SchCom{\fat{[}\ \fat{\nu_0}\ \fat{]}\ }{\fat{[}\ \fat{\pi_0} \ \fat{]}}
\ComSep{}{} 
\cdots 
\ComSep{}{}
\SchCom{\fat{[}\ \fat{\nu_m}\ \fat{]}\ }{\fat{[}\ \fat{\pi_m} \ \fat{]}}\]
such that $\nu_0\neq\top$ is the leading component, then there exists a proof schema $\left\langle \mathbf{C}_{0},\mathbf{C}_{1},\cdots,\mathbf{C}_{k}\right\rangle$, for $k\leq m$, where  for every $0\leq i \leq k$ there exists $0\leq i \leq j \leq m$, where the end sequents of $C_{i}.2$ and $\pi_j$ match as well as the end sequents of $C_{i}.3$ and $\nu_j$.
\end{theorem}
\begin{proof}
By Lemma~\ref{lem:PPSNF} we know that $\Phi$ has a pre-proof schema normal form $\Phi'$. Note that, in a pre-proof schema normal form the leading component, i.e.\ $\nu_0$, is the leftmost component. In $\Phi'$, we delete all component groups whose stepcase is equal to $\fat{[}\ \fat{]}$ and get $\Psi$ which contains $k$ component groups. This is allowed because $\curvearrowright$ cannot link to component groups with stepcase $\fat{[}\ \fat{]}$. We construct the proof schema directly from $\Psi$ where each proof schema component corresponds to a component group of $\Psi$. A proof schema component $\mathbf{C}_i$ is constructed from a component proof $\mathcal{CG}_j=\SchCom{\fat{[}\ \fat{\nu_j}\ \fat{]}\ }{\fat{[}\ \fat{\pi_j} \ \fat{]}}$ as follows: $\mathbf{C}_i.2$ is the proof containing all rules that are identified with $\mathcal{CG}_j$ and that are applied at the top of $cl_{bc}$. $\mathbf{C}_i.3$ is the proof containing all rules that are identified with $\mathcal{CG}_j$ and that are between $cl_{bc}$ and $cl_{sc}$. We translate each component group according to the order of the pre-proof schema normal form, i.e.\ from right to left, to a proof schema component and construct thereby the proof schema of the theorem.
\end{proof}
\section{Properties of the Calculus}
\label{sec:SoundComplete}

In this section we discuss the decision problem for validity, soundness of the calculus, and completeness with respect to $k$-simple proof schemata.

\subsection{Decidability}
Following the formalization of our calculus, we can state a semi-decidability theorem. This is because we distinguish even between component collections where the leading component is equal but there is a variation in the other components.
\begin{theorem}
Let $\fat{\Pi}$ be a collection of closed components that has a $\mathcal{S}\mathit{i}\mathbf{LK}$-proof then we find the proof in a finite number of inferences.
\end{theorem}
\begin{proof}
By Lemma~\ref{lem:PPSNF} we can construct proofs from right to left. In general, the basecase is an \LKE-sequent that is itself semi-decidable. The rightmost component cannot contain any $\curvearrowright$-inferences in the stepcase such that it behaves as an \LKE-proof plus an additional theory axiom for the $\circlearrowright$-rule and is, therefore, semi-decidable. In the next component's stepcase, we consider all $\curvearrowright$-rules again as theory axioms, such that we end up in a semi-decidable fragment again. By the finite number of components the semi-decidability of $\fat{\Pi}$ follows. 
\end{proof}
The more interesting decidability property is of course whether we are able to extend the number of components on the right of a given component such that the new collection of components has a $\mathcal{S}\mathit{i}\mathbf{LK}$-proof. To see that this is not even semi-decidable we will formalize Robinson arithmetic \cite{RRobinson1950} in our system. 

\begin{theorem}
Let $C$ be a closed component group. Then deciding if there exists a closed component collection $\fat{\Pi}$ such that  $C \ComSep{}{} \fat{\Pi}$ is  $\mathcal{S}\mathit{i}\mathbf{LK}$-provable is undecidable. 
\end{theorem}
\begin{proof}
The $\omega$ sort obeys the axioms of Robinson arithmetic concerning successor and zero. We can add the addition and multiplication axioms to the equational theory. The most important axiom of Robinson arithmetic $\forall x ( x=0 \vee \exists y ( s(y) = x))$ is intrinsically part of the link mechanism. Because  Robinson arithmetic is essentially undecidable then showing that there is an extension of a given component collection to an $\mathcal{S}\mathit{i}\mathbf{LK}$-provable collection must be as well.
\end{proof}

\subsection{Soundness \& Completeness}
We provide a proof of soundness using our translation procedure of Section~\ref{Sec:Translation}.
\begin{theorem}[Soundness of the $\mathcal{S}\mathit{i}\mathbf{LK}$-calculus]
If a closed component collection $C$ is $\mathcal{S}\mathit{i}\mathbf{LK}$-provable then it is valid. 
\end{theorem}
\begin{proof}
Let $\Phi$ be an $\mathcal{S}\mathit{i}\mathbf{LK}$-proof of $C$. By Section~\ref{Sec:Translation} we can transform $\Phi$ to a pre-proof schema normal form $\Phi'$ and then construct a proof schema $\Psi$ from it. By Proposition~\ref{prop:sound}, we show the validity of the leading component and, therefore, of the evaluation itself, i.e.\ the $\mathcal{S}\mathit{i}\mathbf{LK}$-calculus is sound.
\end{proof}

To show completeness we technically need a conversion from proof schemata to $\mathcal{S}\mathit{i}\mathbf{LK}$-proofs which can be easily derived given the procedure defined in Section~\ref{Sec:Translation}. Due to space constraints we avoid
formally defining the procedure. 
\begin{theorem}[Completeness]
If a close component collection $C$ represents a valid $n$-induction statement then it is  $\mathcal{S}\mathit{i}\mathbf{LK}$-provable. 
\end{theorem}
\begin{proof}
By the theorems and definitions of Section~\ref{sec:evalInter}, we know that if $C$ represents a valid $n$-induction statement then a proof can be found in the $\mathbf{LKIE}$-calculus. Any $\mathbf{LKIE}$-proof can be transformed into a proof schema $\Phi$ (Section~\ref{sec:evalInter}). We have not shown that $\Phi$ can be transformed into $\mathcal{S}\mathit{i}\mathbf{LK}$-proofs, but it is quite obvious that the procedure defined in Section~\ref{Sec:Translation} is reversible. Thus, there is a $\mathcal{S}\mathit{i}\mathbf{LK}$-proof for $C$. 
\end{proof}

\section{Conclusion}
\label{sec:Conclusion}

In this paper we introduce a calculus for the construction of proof schemata which we refer to as the calculus $\mathcal{S}\mathit{i}\mathbf{LK}$. Initially, proof schemata were formalized by first defining an extension of the calculus \textbf{LK}, the calculus \textbf{LKS}, which adds so called links to the set of allowed axioms and an equational theory rule. Using this extended calculus a formal definition for proof schemata was developed~\cite{CDunchev2014}. The problem with this approach is that the calculus \textbf{LKS} is only a calculus for defining proof schemata not a calculus for proof schemata. Prior to the work described in this paper a calculus for proof schemata did not exists. 


\bibliographystyle{plain}
\bibliography{References}

\appendix
\section{Appendix}
\label{appendix}

If we refer to the \LK-calculus we mean the following version:
\begin{center}
\begin{tabular}{c}
\begin{minipage}{0.45\textwidth}
\begin{small}
\begin{prooftree}
\AxiomC{$A,B,\Gamma\vdash\Delta$}
\RightLabel{$\wedge :l$}
\UnaryInfC{$A\wedge B,\Gamma\vdash\Delta$}
\end{prooftree}
\end{small}
\end{minipage}
\begin{minipage}{0.45\textwidth}
\begin{small}
\begin{prooftree}
\AxiomC{$\Gamma\vdash\Delta ,A$}
\AxiomC{$\Pi\vdash\Lambda ,B$}
\RightLabel{$\wedge :r$}
\BinaryInfC{$\Gamma ,\Pi\vdash\Delta ,\Lambda ,A\wedge B$}
\end{prooftree}
\end{small}
\end{minipage}
\\
\begin{minipage}{0.45\textwidth}
\begin{small}
\begin{prooftree}
\AxiomC{$A,\Gamma\vdash\Delta$}
\AxiomC{$B,\Pi\vdash\Lambda$}
\RightLabel{$\vee :l$}
\BinaryInfC{$A\vee B,\Gamma ,\Pi\vdash\Delta ,\Lambda$}
\end{prooftree}
\end{small}
\end{minipage}
\begin{minipage}{0.45\textwidth}
\begin{small}
\begin{prooftree}
\AxiomC{$\Gamma\vdash\Delta ,A,B$}
\RightLabel{$\vee :r$}
\UnaryInfC{$\Gamma\vdash\Delta ,A\vee B$}
\end{prooftree}
\end{small}
\end{minipage}
\\
\begin{minipage}{0.45\textwidth}
\begin{small}
\begin{prooftree}
\AxiomC{$A,\Gamma\vdash\Delta$}
\RightLabel{$\neg :l$}
\UnaryInfC{$\Gamma\vdash\Delta ,\neg A$}
\end{prooftree}
\end{small}
\end{minipage}
\begin{minipage}{0.45\textwidth}
\begin{small}
\begin{prooftree}
\AxiomC{$\Gamma\vdash\Delta ,A$}
\RightLabel{$\neg :r$}
\UnaryInfC{$\neg A ,\Gamma\vdash\Delta$}
\end{prooftree}
\end{small}
\end{minipage}
\\
\begin{minipage}{0.45\textwidth}
\begin{small}
\begin{prooftree}
\AxiomC{$\Gamma\vdash\Delta ,A$}
\AxiomC{$B, \Pi\vdash\Lambda$}
\RightLabel{$\to :l$}
\BinaryInfC{$A\to B,\Gamma ,\Pi\vdash\Delta ,\Lambda$}
\end{prooftree}
\end{small}
\end{minipage}
\begin{minipage}{0.45\textwidth}
\begin{small}
\begin{prooftree}
\AxiomC{$A,\Gamma\vdash\Delta ,B$}
\RightLabel{$\to :r$}
\UnaryInfC{$\Gamma\vdash\Delta ,A\to B$}
\end{prooftree}
\end{small}
\end{minipage}
\\
\begin{minipage}{0.3\textwidth}
\begin{small}
\begin{prooftree}
\AxiomC{$ \ $}
\RightLabel{$Ax$}
\UnaryInfC{$A\vdash A$}
\end{prooftree}
\end{small}
\end{minipage}
\begin{minipage}{0.3\textwidth}
\begin{small}
\begin{prooftree}
\AxiomC{$A,A,\Gamma\vdash\Delta$}
\RightLabel{$c:l$}
\UnaryInfC{$A,\Gamma\vdash\Delta$}
\end{prooftree}
\end{small}
\end{minipage}
\begin{minipage}{0.3\textwidth}
\begin{small}
\begin{prooftree}
\AxiomC{$\Gamma\vdash\Delta ,A,A$}
\RightLabel{$c:r$}
\UnaryInfC{$\Gamma\vdash\Delta ,A$}
\end{prooftree}
\end{small}
\end{minipage}
\\
\begin{minipage}{0.4\textwidth}
\begin{small}
\begin{prooftree}
\AxiomC{$\Gamma\vdash\Delta ,A$}
\AxiomC{$A,\Pi\vdash\Lambda$}
\RightLabel{$cut$}
\BinaryInfC{$\Gamma ,\Pi\vdash\Delta ,\Lambda$}
\end{prooftree}
\end{small}
\end{minipage}
\begin{minipage}{0.25\textwidth}
\begin{small}
\begin{prooftree}
\AxiomC{$\Gamma\vdash\Delta$}
\RightLabel{$w:l$}
\UnaryInfC{$A,\Gamma\vdash\Delta$}
\end{prooftree}
\end{small}
\end{minipage}
\begin{minipage}{0.25\textwidth}
\begin{small}
\begin{prooftree}
\AxiomC{$\Gamma\vdash\Delta$}
\RightLabel{$w:r$}
\UnaryInfC{$\Gamma\vdash\Delta ,A$}
\end{prooftree}
\end{small}
\end{minipage}
\\
\begin{minipage}{0.45\textwidth}
\begin{small}
\begin{prooftree}
\AxiomC{$A[x\backslash\alpha ],\Gamma\vdash\Delta$}
\RightLabel{$\forall :l$}
\UnaryInfC{$\forall x.A,\Gamma\vdash\Delta$}
\end{prooftree}
\end{small}
\end{minipage}
\begin{minipage}{0.45\textwidth}
\begin{small}
\begin{prooftree}
\AxiomC{$\Gamma\vdash\Delta ,A[x\backslash t]$}
\RightLabel{$\forall :r$}
\UnaryInfC{$\Gamma\vdash\Delta ,\forall x.A$}
\end{prooftree}
\end{small}
\end{minipage}
\\
\begin{minipage}{0.45\textwidth}
\begin{small}
\begin{prooftree}
\AxiomC{$A[x\backslash t],\Gamma\vdash\Delta$}
\RightLabel{$\exists :l$}
\UnaryInfC{$\exists x.A,\Gamma\vdash\Delta$}
\end{prooftree}
\end{small}
\end{minipage}
\begin{minipage}{0.45\textwidth}
\begin{small}
\begin{prooftree}
\AxiomC{$\Gamma\vdash\Delta ,A[x\backslash\alpha ]$}
\RightLabel{$\exists :r$}
\UnaryInfC{$\Gamma\vdash\Delta ,\exists x.A$}
\end{prooftree}
\end{small}
\end{minipage}
\end{tabular}
\end{center}
where $\alpha$ is a fresh variable not occurring in $\Gamma ,\Delta ,$ or $A$. 
\end{document}